\documentclass[reqno,10pt,centertags,draft]{amsart}
\usepackage{amsmath,amsthm,amscd,amssymb,latexsym,upref,stmaryrd}

\newcommand{\bbC}{{\mathbb{C}}}

\newcommand{\bbN}{{\mathbb{N}}}

\newcommand{\bbR}{{\mathbb{R}}}

\newcommand{\bfS}{{\mathbf{S}}}
\newcommand{\bfT}{{\mathbf{T}}}

\newcommand{\bsA}{{\boldsymbol{A}}}
\newcommand{\bsB}{{\boldsymbol{B}}}

\newcommand{\bsP}{{\boldsymbol{P}}}

\newcommand{\bsS}{{\boldsymbol{S}}}
\newcommand{\bsT}{{\boldsymbol{T}}}

\newcommand{\cB}{{\mathcal B}}

\newcommand{\cD}{{\mathcal D}}
\newcommand{\cE}{{\mathcal E}}

\newcommand{\cH}{{\mathcal H}}

\newcommand{\cM}{{\mathcal M}}

\newcommand{\cS}{{\mathcal S}}
\newcommand{\cT}{{\mathcal T}}

\newcommand{\cY}{{\mathcal Y}}

\newcommand{\gE}{\mathfrak{E}}

\DeclareMathOperator{\dom}{dom}
\DeclareMathOperator{\codim}{codim}

\renewcommand{\Im}{\text{\rm Im}}

\newcommand{\beq}{\begin{equation}}
\newcommand{\enq}{\end{equation}}

\newcommand{\no}{\notag}
\newcommand{\lb}{\label}
\newcommand{\f}{\frac}

\newcommand{\ol}{\overline}

\newcommand{\wti}{\widetilde}

\newcommand{\hatt}{\widehat}
\newcommand{\bi}{\bibitem}

\let\geq\geqslant
\let\leq\leqslant

\numberwithin{equation}{section}

\allowdisplaybreaks \numberwithin{equation}{section}

\newtheorem{theorem}{Theorem}[section]

\newtheorem{lemma}[theorem]{Lemma}
\newtheorem{corollary}[theorem]{Corollary}
\newtheorem{hypothesis}[theorem]{Hypothesis}
\newtheorem{definition}[theorem]{Definition}

\newtheorem{example}[theorem]{Example}
\theoremstyle{remark}
\newtheorem{remark}[theorem]{Remark}

\begin{document}

\numberwithin{equation}{section}
\allowdisplaybreaks

\title[On a Question of A.\ E.\ Nussbaum]{On a Question of A.\ E.\ Nussbaum on
Measurability of Families of Closed Linear Operators in a Hilbert Space}

\author[F.\ Gesztesy]{Fritz Gesztesy}
\address{Department of Mathematics,
University of Missouri, Columbia, MO 65211, USA}
\email{gesztesyf@missouri.edu}
\urladdr{http://www.math.missouri.edu/personnel/faculty/gesztesyf.html}

\author[A.\ Gomilko]{Alexander Gomilko}
\address{Faculty of Mathematics and Computer Science, Nicholas
Copernicus University, ul.
Chopina 12/18, 87-100 Torun, Poland, and Institute of Mathematics, Polish Academy of Sciences. \'Sniadeckich str. 8, 00-956 Warsaw, Poland}
\email{gomilko@mat.uni.torun.pl}

\author[F.\ Sukochev]{Fedor Sukochev}
\address{School of Mathematics and Statistics, UNSW, Kensington, NSW 2052,
Australia}
\email{f.sukochev@unsw.edu.au}

\author[Y.\ Tomilov]{Yuri Tomilov}
\address{Faculty of Mathematics and Computer Science, Nicholas
Copernicus University, ul.
Chopina 12/18, 87-100 Torun, Poland, and Institute of Mathematics, Polish Academy of Sciences. \'Sniadeckich str. 8, 00-956 Warsaw, Poland}
\email{tomilov@mat.uni.torun.pl}

\dedicatory{Dedicated to the memory of A.\ E.\ Nussbaum (1925--2009)}
\thanks{A.G. and Y.T. were partially supported by the  
Marie Curie ''Transfer of Knowledge'' programme, project ``TODEQ''.  
Y.T. was also partially supported by a MNiSzW grant Nr.\ N201384834. The work 
of F.S. was partially supported by the ARC.}

\date{\today}

\begin{abstract} 
The purpose of this note is to answer a question A.\ E.\ Nussbaum 
formulated in 1964 about the possible equivalence between weak measurability of 
a family of densely defined, closed operators $\{T(t)\}_{t\in\bbR}$ in a 
separable complex Hilbert space $\cH$ on one hand, and the notion of 
measurability of the 
$2 \times 2$ operator-valued matrix of projections 
$\big\{\big(P(\Gamma(T(t)))_{j,k}\big)_{1\leq j,k \leq 2}\big\}_{t\in\bbR}$ 
onto the graph $\Gamma(T(t))$ of $T(t)$ on the other, in the negative. 

We also consider related questions pertaining to the family of adjoint 
operators $\{T(t)^*\}_{t\in\bbR}$ and to the issue of whether or not the 
corresponding maximally defined operator $\boldsymbol{\cT}$ in $L^2(\bbR; dt;\cH)$ 
given by 
\begin{align*}
&(\boldsymbol{\cT} f)(t) = T(t) f(t), \quad t\in\bbR,   \no \\
& f \in \dom(\boldsymbol{\cT}) = \bigg\{g \in L^2(\bbR;dt;\cH) \,\bigg|\,
g(t)\in \dom(T(t)) \text{ for a.e.\ } t\in\bbR,  \no  \\
& \quad t \mapsto T(t)g(t) \text{ is weakly measurable,} \,
\int_{\bbR} \|T(t) g(t)\|_{\cH}^2 \, dt <
\infty\bigg\}.   \no
\end{align*} 
is densely defined. Our results demonstrate an interesting 
distinction between the operator $\boldsymbol{\cT}$ and the direct integral
$\bsT = \int_{\bbR}^{\oplus} T(t) \, dt$ in $L^2(\bbR; \cH)$ (the latter requires 
the additional assumption of measurability of the matrix of projections 
$\big\{\big(P(\Gamma(T(t)))_{j,k}\big)_{1\leq j,k \leq 2}\big\}_{t\in\bbR}$.

We also provide explicit criteria for the measurability of the matrix of projections 
$\big\{\big(P(\Gamma(T(t)))_{j,k}\big)_{1\leq j,k \leq 2}\big\}_{t\in\bbR}$.
\end{abstract}

\maketitle

\section{Introduction}  \lb{s1}

To briefly set the stage for this note, let $\cH$ be a separable complex Hilbert 
space, and consider the Hilbert space $L^2(\bbR; dt; \cH)$, in short, $L^2(\bbR; \cH)$, consisting of equivalence classes $f$ of weakly (and hence strongly) Lebesgue measurable $\cH$-valued elements $f(\cdot)\in\cH$ (whose elements are equal 
a.e.\ on $\bbR$), such that $\|f(\cdot)\|_{\cH} \in L^2(\bbR; dt)$. 
Of course, $L^2(\bbR; \cH)$ can be identified with the constant fiber
direct integral $\int_{\bbR}^{\oplus} \cH \, dt$, that is,
\beq
L^2(\bbR; \cH) = \int_{\bbR}^{\oplus} \cH \, dt.    \lb{1.1}
\enq

Let $\bbR \ni t \mapsto g(t)\in \cH$, then a family $\{g(t)\}_{t\in\bbR}$ is 
called {\it weakly measurable} in $\cH$ if $\bbR \ni t \mapsto (h, g(t))_{\cH}$
is $($Lebesgue$)$ measurable for each $h \in \cH$. 

Throughout this manuscript, boldface calligraphic symbols, such as 
$\boldsymbol{\cS}$, denote operators 
in the Hilbert space $L^2(\bbR;\cH)$ associated with a
family of linear operators $\{S(t)\}_{t\in\bbR}$ in $\cH$, maximally defined by 
\begin{align}
&(\boldsymbol{\cS} f)(t) = S(t) f(t) \, \text{ for a.e.\ $t\in\bbR$,}   \no \\
& f \in \dom(\boldsymbol{\cS}) = \bigg\{g \in L^2(\bbR;\cH) \,\bigg|\,
g(t)\in \dom(S(t)) \text{ for a.e.\ } t\in\bbR,    \lb{1.2}  \\
& \quad t \mapsto S(t)g(t) \text{ is (weakly) measurable,} \,
\int_{\bbR} \|S(t) g(t)\|_{\cH}^2 \, dt <\infty\bigg\}.   \no
\end{align}
One readily verifies that if $S(t)$ are closed in $\cH$ for a.e.\ $t\in\bbR$, then 
$\boldsymbol{\cS}$ is closed in $L^2(\bbR;\cH)$.

Next, let $\{T(t)\}_{t\in\bbR}$ be a family of densely defined, closed, linear 
operators in $\cH$. Then the family $\{T(t)\}_{t\in\bbR}$ is called 
{\it weakly measurable} if for any weakly measurable family of elements 
$\{f(t)\}_{t\in\bbR}$ in $\cH$ such that
$f(t) \in \dom(T(t))$ for a.e.\ $t\in\bbR$, the family of elements $\{T(t) f(t)\}_{t\in\bbR}$
 is weakly measurable in $\cH$. 
 
Given the family $\{T(t)\}_{t\in\bbR}$, one defines 
the operator $\boldsymbol{\cT}$ in $L^2(\bbR; \cH)$ as in \eqref{1.2}. Then 
$\boldsymbol{\cT}$ is closed but 
not necessarily densely defined. In order to relate $\boldsymbol{\cT}$ with the direct integral over the operators $T(t)$, $t\in\bbR$, with respect to Lebesgue measure, Nussbaum \cite{Nu64} introduces the following fundamental notion of measurability of the family $\{T(t)\}_{t\in\bbR}$, which we will call $N$-measurability in his honor.  

Denote by $\{(P(\Gamma(T(t))))_{1\leq j,k \leq 2}\}_{t\in\bbR}$ the $2 \times 2$ 
operator-valued matrix of projections onto the graph $\Gamma(T)$ of $T(t)$. 
Then the family $\{T(t)\}_{t\in\bbR}$ is called {\it $N$-measurable} if 
$\big\{(P(\Gamma(T(t)))_{j,k}\big\}_{t\in\bbR}$, $j,k \in \{1, 2\}$, are weakly measurable. 
   
Assuming the family $\{T(t)\}_{t\in\bbR}$ to be $N$-measurable, the operator 
$\boldsymbol{\cT}$ is called 
{\it decomposable} in $L^2(\bbR; \cH) = \int_{\bbR}^{\oplus} \cH \, dt$ and also 
denoted by the direct integral of the family $\{T(t)\}_{t\in\bbR}$ over $\bbR$ with 
respect to Lebesgue measure,
\begin{equation}
\bsT = \int_{\bbR}^{\oplus} T(t) \, dt.    \lb{1.3}
\end{equation}

Throughout this manuscript, boldface symbols, such as $\bsT$, denote 
operators in the Hilbert space $L^2(\bbR;\cH)$ associated with the direct 
integral over the family $\{T(t)\}_{t\in\bbR}$ as depicted in \eqref{1.3} (in 
contrast to our choice of notation $\boldsymbol{\cT}$ in the context of 
\eqref{1.2}). 

Given these preparations, we can now attempt to formulate the question posed by Nussbaum \cite[p.\ 36]{Nu64}.  

In the special case of bounded operators $T(t) \in \cB(\cH)$, $t\in\bbR$, 
Nussbaum \cite{Nu64} proved the equivalence of $N$-measurability and weak measurability of the family $\{T(t)\}_{t\in\bbR}$. Moreover, he also proved that 
in general, $N$-measurability of $\{T(t)\}_{t\in\bbR}$ implies weak measurability 
of $\{T(t)\}_{t\in\bbR}$. However, in the case of unbounded operators $T(t)$, the converse of this fact, and hence the following version of Nussbaum's question:  
\begin{equation}
\text{$\bullet$ Is $N$-measurability of $\{T(t)\}_{t\in\bbR}$ equivalent to 
weak measurability of $\{T(t)\}_{t\in\bbR}$\,?}    \lb{1.4}
\end{equation}
appears to have been open since 1964.

The principal purpose of this note is to answer Nussbaum's question in the negative
and hence demonstrate an interesting distinction between operators 
$\boldsymbol{\cT}$ in $L^2(\bbR; \cH)$ defined according to \eqref{1.2} on one hand, and direct integrals $\bsT = \int_{\bbR}^{\oplus} T(t) \, dt$ as in \eqref{1.3}, on the other hand, in the sense that it may happen that $\boldsymbol{\cT}$ exists but that 
$\boldsymbol{\cT}$ cannot be identified with $\int_{\bbR}^{\oplus} T(t) \, dt$.

In addition, we also answer the following natural question: 
\begin{align}
\begin{split}
& \text{$\bullet$ Is weak measurability of $\{T(t)\}_{t\in \bbR}$ equivalent to weak 
measurability} \\
& \;\;\; \text{of $\{T(t)^*\}_{t\in \bbR}$\,?}    \lb{1.5}
\end{split}
\end{align}
in the negative (thereby independently answering Nussbaum's question 
\eqref{1.4} in the negative once again). Finally, we also address the question of 
whether operators of the type $\boldsymbol{\cT}$ are densely defined in 
$L^2(\bbR; \cH)$. 

In Section \ref{s2} we very briefly recall basic facts on closed operators and their graphs following Stone's fundamental paper \cite{St51}. Fundamental facts for direct integrals of (unbounded) closed operators as developed in Nussbaum \cite{Nu64} (see also 
Lennon \cite{Le74} and Pallu de la Barri\`ere \cite{Pa51}) are summarized in 
Section \ref{s3}. In our final Section \ref{s4} we present counterexamples to questions 
\eqref{1.3} and \eqref{1.4}, investigate when $\boldsymbol{\cT}$ is densely defined or not, and conclude with a sufficient criterion for a weakly measurable family 
$\{T(t)\}_{t\in \bbR}$ to be $N$-measurable in terms of the resolvents of $T(t)$, 
$t\in\bbR$.

Finally, we briefly summarize some of the notation used in this paper: Let $\cH$ be a
separable complex Hilbert space, $(\cdot,\cdot)_{\cH}$ the scalar product in $\cH$
(linear in the second argument), and $I_{\cH}$ the identity operator in $\cH$.
Next, let $T$ be a linear operator mapping (a subspace of) a
Hilbert space into another, with $\dom(T)$ and $\ker(T)$ denoting the
domain and kernel (i.e., null space) of $T$. 
The resolvent set 
of a closed linear operator in $\cH$ will be denoted by $\rho(\cdot)$. 
The Banach space of bounded linear operators on $\cH$ is
denoted by $\cB(\cH)$.

\medskip

\section{Some Facts on Closed Linear Operators}  \lb{s2}

The principal purpose of this short section is to briefly recall some basic facts on 
closed operators and their graphs discussed in great detail in Stone's fundamental 
paper \cite{St51} (see also von Neumann \cite[Ch.\ XIII, App.\ II]{vN51}).

For simplicity, we make the following assumption:

\begin{hypothesis} \lb{hA.-1}
Let $\cH$ be a complex separable Hilbert space and $T$ a densely defined, closed,
linear operator in $\cH$.
\end{hypothesis}

We note that Stone \cite{St51} considers a more general situation, but Hypothesis
\ref{hA.-1} fits the purpose of our paper.

By $\Gamma (T)$ we denote the graph of
$T$, that is, the following subspace of the direct sum $\cH \oplus \cH$,
\begin{equation}
\Gamma (T) = \{\langle f, Tf\rangle \,|\, f \in \dom(T)\} \subseteq \cH \oplus \cH.  \lb{A.1}
\end{equation}
Since $T$ is assumed to be closed, $\Gamma (T)$ is a closed subspace of
$\cH \oplus \cH$. Here $\langle f, g \rangle$ denotes the ordered pair of $f, g \in \cH$,  and we use the standard norm
\begin{equation}
\|\langle f, g\rangle\|_{\cH\oplus \cH} = \big[\|f\|^2_{\cH} + \|g\|^2_{\cH}\big]^{1/2},
\quad f, g \in \cH,    \lb{A.2}
\end{equation}
and scalar product
\begin{equation}
(\langle f_1, g_1\rangle, \langle f_2, g_2\rangle)_{\cH\oplus \cH}
= (f_1,f_2)_{\cH} + (g_1,g_2)_{\cH},   \quad f_j, g_j \in \cH, \; j=1,2,     \lb{A.3}
\end{equation}
in $\cH\oplus\cH$.

If $B \in \cB(\cH\oplus\cH)$, one can uniquely represent $B$ as the $2\times 2$ block operator matrix
\begin{equation}
B = \begin{pmatrix} B_{1,1} & B_{1,2} \\ B_{2,1} & B_{2,2} \end{pmatrix},  \lb{A.4}
\end{equation}
where $B_{j,k} \in \cB(\cH)$, $j, k \in \{1,2\}$.

Denoting by
\begin{equation}
P(\Gamma(T)) =  \begin{pmatrix} P(\Gamma(T))_{1,1} 
& P(\Gamma(T))_{1,2} \\
P(\Gamma(T))_{2,1} & P(\Gamma(T))_{2,2} \end{pmatrix}    \lb{A.5}
\end{equation} 
the orthogonal projection onto $\Gamma (T)$, the corresponding matrix
$(P(\Gamma(T))_{j,k})_{1\leq j,k \leq 2}$ will be called the {\it characteristic matrix} of $T$.
Since by hypothesis $T$ is closed and densely defined, one obtains (cf.\ 
\cite{St51})
\begin{align}
& (P(\Gamma(T))_{j,k})^* = P(\Gamma(T))_{k,j}, \quad j, k \in \{1,2\},   \lb{A.6} \\
& \sum_{k=1}^2 P(\Gamma(T))_{j,k} P(\Gamma(T))_{k,\ell} 
= P(\Gamma(T))_{j,\ell}, \quad j, \ell \in \{1,2\},   \lb{A.7} \\
& \ker (P(\Gamma(T))_{1,1}) = \ker (I_{\cH} - P(\Gamma(T))_{2,2}) = \{0\},    \lb{A.8} \\
& P(\Gamma(T^*)) = \begin{pmatrix} I_{\cH} - P(\Gamma(T))_{2,2} 
& P(\Gamma(T))_{2,1} \\
P(\Gamma(T))_{1,2} & I_{\cH} - P(\Gamma(T))_{1,1} \end{pmatrix},  \lb{A.9} \\
& \ker (T) = \{0\} \, \text{ if and only if } \, \ker (I_{\cH} - P(\Gamma(T))_{1,1}) = \{0\},   
\lb{A.10} \\
& P\big(\Gamma\big(T^{-1}\big)\big) =  \begin{pmatrix} P(\Gamma(T))_{2,2} 
& P(\Gamma(T))_{2,1} \\
P(\Gamma(T))_{1,2} & P(\Gamma(T))_{1,1} \end{pmatrix} \, \text{ if } \, \ker (T) = \{0\},   
\lb{A.11}  \\
& P(\Gamma(T))_{2,1} = T P(\Gamma(T))_{1,1}, 
\quad P(\Gamma(T))_{2,2} = T P(\Gamma(T))_{1,2},    \lb{A.12} \\
& I_{\cH} - P(\Gamma(T))_{1,1} = T^* P(\Gamma(T))_{2,1}, \quad
P(\Gamma(T))_{1,2} = T^* (I_{\cH} - P(\Gamma(T))_{2,2}).    \lb{A.13}
\end{align}
In particular, one has the following explicit expressions for 
$P(\Gamma(T))_{j,k}$, $j,k \in \{1,2\}$:
\begin{align}
\begin{split}
& P(\Gamma(T))_{1,1} = (T^* T + I_{\cH})^{-1},    \\
& P(\Gamma(T))_{1,2} = T^* (T T^* + I_{\cH})^{-1},   \\
& P(\Gamma(T))_{2,1} = T (T^* T + I_{\cH})^{-1},   \lb{A.14}  \\
& P(\Gamma(T))_{2,2} = T T^* (T T^* + I_{\cH})^{-1} = I_{\cH} - (T T^* + I_{\cH})^{-1}
\end{split}
\end{align}
(see also \cite{ABJT09}, \cite{HSDS07}, and \cite{Me00} for generalizations to closed linear relations).  

\section{Basic Facts on Direct Integrals of Closed Operators}  \lb{s3}

We briefly recall some facts for direct integrals of unbounded closed operators as 
developed in Nussbaum \cite{Nu64} (see also Dixon \cite{Di71}, 
Lennon \cite{Le74}, Pallu de la Barri\`ere \cite{Pa51}, and Takemoto 
\cite{Ta75}).

We study families of densely defined, closed operators $\{T(t)\}_{t\in\bbR}$
in $\cH$ and use the following assumption for the remainder of this section:

\begin{hypothesis} \lb{hA.0}
Let $T(t)$, $t\in\bbR$, be densely defined, closed, linear operators in $\cH$.
\end{hypothesis}

We need the following notions of measurable vector and operator families:

\begin{definition} \lb{dA.1}
$(i)$ Let $\bbR \ni t \mapsto g(t)\in \cH$. Then the family $\{g(t)\}_{t\in\bbR}$
is called {\it weakly measurable} in $\cH$ if $\bbR \ni t \mapsto (h, g(t))_{\cH}$
is $($Lebesgue$)$ measurable for each $h \in \cH$. \\
Next, assume Hypothesis \ref{hA.0}: \\
$(ii)$ The family $\{T(t)\}_{t\in\bbR}$ is called {\it weakly measurable} if for any
weakly measurable family of elements $\{f(t)\}_{t\in\bbR}$ in $\cH$ such that
$f(t) \in \dom(T(t))$ for all $t\in\bbR$, the family of elements $\{T(t) f(t)\}_{t\in\bbR}$ is weakly measurable in $\cH$. \\
$(iii)$ The family $\{T(t)\}_{t\in\bbR}$ is called {\it $N$-measurable} if the entries of the characteristic matrix of $T(t)$ are weakly measurable, that is, if
$\big\{P(\Gamma(T(t)))_{j,k}\big\}_{t\in\bbR}$, $j,k \in \{1, 2\}$, are weakly measurable.
\end{definition}

For notational simplicity the vector and operator families in this note are defined for all 
$t\in\bbR$ rather than the customary a.e.\ $t\in\bbR$ with respect to Lebesgue 
measure.  


We note that measurability of the characteristic matrix
$(P(\Gamma(T(\cdot)))_{j,k})_{1\leq j,k \leq 2}$ of $T(\cdot)$ was introduced by
Nussbaum \cite{Nu64}. In fact, he considered the more general situation of a general measure $d\mu$ and a $\mu$-measurable family of Hilbert spaces 
$\{\cH(t)\}_{t\in\bbR}$.

We summarize a few consequences of Definition \ref{dA.1} in Remark \ref{rA.2} below:

\begin{remark}  \lb{rA.2}
$(i)$ Since $\cH$ is assumed to be separable,
weak measurability of the family  $\{g(t)\}_{t\in\bbR}$ in $\cH$
is equivalent to (strong) measurability, that is, there exists a sequence of
countably-valued elements $\{g_n(t)\}_{t\in\bbR} \subset \cH$,
$n\in\bbN$, and a set $\cE \subset \bbR$ of Lebesgue measure zero
such that $\lim_{n\to\infty} \|g_n(t) - g(t)\|_{\cH} =0$ for each
$t\in \bbR\backslash \cE$. Thus, the family $\{g(t)\}_{t\in\bbR}$ is  
(weakly) measurable in $\cH$ if there exists a dense set $\cY \subset \cH$
such that the function $(y, g(\cdot))_{\cH}$ is measurable for every $y \in \cY$, see,
for instance, \cite[Corollary\ 1.1.3]{ABHN01}, \cite[p.\ 42--43]{DU77}. Moreover,
\begin{equation}
f,g: \bbR \mapsto \cH \, \text{ measurable } \,
\Longrightarrow \, (f(\cdot), g(\cdot))_{\cH} \, \text{ is measurable.}   \lb{A.14aa}
\end{equation}
$(ii)$ One can show (cf.\ \cite[Corollary\ 2]{Nu64}) that
\begin{equation}
\text{$N$-measurability of $\{T(t)\}_{t\in\bbR}$ $\Longrightarrow$ weak measurability of
$\{T(t)\}_{t\in\bbR}$}.   \lb{A.14a}
\end{equation}
(The converse, however, is false as we will show in Section \ref{s4}.) \\
$(iii)$ Since by \eqref{A.6}, 
$P(\Gamma(T(t)))_{2,1} = (P(\Gamma(T(t)))_{1,2})^*$, or equivalently, since
\begin{align}
\begin{split}
\big[T(t) (T(t)^* T(t) + I_{\cH})^{-1}\big]^* &= T(t)^* (T(t) T(t)^* + I_{\cH})^{-1} \\
& \supseteq  (T(t)^* T(t) + I_{\cH})^{-1} T(t)^*,
\end{split}
\end{align}
as $T(t)$ is closed in $\cH$ (this follows from standard properties of adjoints of products of linear operators and from \eqref{A.6}, \eqref{A.9}, and \eqref{A.14}; see also 
\cite[Theorem\,2\,(ii)]{De78}), weak measurability of
the operator $\{P(\Gamma(T(t)))_{1,2}\}_{t\in\bbR}$ is equivalent to that of 
$\{P(\Gamma(T(t)))_{2,1}\}_{t\in\bbR}$. Thus, by \eqref{A.14},
\begin{align}
& \text{$N$-measurability of $\{T(t)\}_{t\in\bbR}$ is equivalent to weak measurability of}  \no  \\
& \quad \big\{\big(|T(t)|^2 + I_{\cH}\big)^{-1}\big\}_{t\in\bbR},
\quad \big\{T(t) \big(|T(t)|^2 + I_{\cH}\big)^{-1}\big\}_{t\in\bbR},    \lb{A.14A} \\
& \quad \text{and } \, \big\{\big(|T(t)^*|^2 + I_{\cH}\big)^{-1}\big\}_{t\in\bbR}.    \no
\end{align}
$(iv)$ Items $(ii)$ and $(iii)$ show that
\begin{equation}
\text{$N$-measurability of $\{T(t)\}_{t\in\bbR}$ $\Longleftrightarrow$
$N$-measurability of $\{T(t)^*\}_{t\in\bbR}$.}    \lb{A.14B}
\end{equation}
$(v)$ If $T(t) \in \cB(\cH)$, $t\in\bbR$, then weak measurability of $\{T(t)\}_{t\in\bbR}$
is equivalent to
\begin{equation}
(g, T(t) h)_{\cH} \, \text{ is (Lebesgue) measurable for each $g, h \in \cH$.}  \lb{A.14b}
\end{equation}
Moreover (cf.\ \cite[Proposition\ 6]{Nu64}),
\begin{align}
\begin{split}
& \text{if $T(t) \in \cB(\cH)$ for a.e.\  $t\in\bbR$, then} \\
& \text{weak measurability of $\{T(t)\}_{t\in\bbR}$ $\Longleftrightarrow$
$N$-measurability of $\{T(t)\}_{t\in\bbR}$.}    \lb{A.14c}
\end{split}
\end{align}
\end{remark}

 \smallskip
The Hilbert space $L^2(\bbR; dt; \cH)$, in short, $L^2(\bbR; \cH)$, consists of equivalence 
classes $f$ of weakly (and hence strongly) Lebesgue measurable
$\cH$-valued elements $f(\cdot)\in\cH$ (whose elements are equal a.e.\ on $\bbR$), 
such that $\|f(\cdot)\|_{\cH} \in L^2(\bbR; dt)$. The norm and scalar product on
$L^2(\bbR; \cH)$ are then given by
\beq
  \|f\|_{L^2(\bbR; \cH)}^2 = \int_{\bbR}  \|f(t)\|_{\cH}^2 \, dt, \;\;
  (f,g)_{L^2(\bbR; \cH)} = \int_{\bbR}  (f(t), g(t))_{\cH} \, dt, \;\;
  f, g \in L^2(\bbR; \cH).   \lb{A.15}
\enq

Of course, $L^2(\bbR; \cH)$ can be identified with the constant fiber
direct integral $\int_{\bbR}^{\oplus} \cH \, dt$ (cf., e.g., \cite[Sect.\ II.1]{Di96}, 
\cite[Ch.\ XII]{vN51}), that is,
\beq
L^2(\bbR; \cH) = \int_{\bbR}^{\oplus} \cH \, dt.    \lb{A.16}
\enq

We recall our convention that throughout this manuscript, boldface calligraphic 
symbols, such as $\boldsymbol{\cS}$, denote operators in the Hilbert space 
$L^2(\bbR;\cH)$ associated with a
family of linear operators $\{S(t)\}_{t\in\bbR}$ in $\cH$, maximally defined by
\begin{align}
&(\boldsymbol{\cS} f)(t) = S(t) f(t) \, \text{ for a.e.\ $t\in\bbR$,}   \no \\
& f \in \dom(\boldsymbol{\cS}) = \bigg\{g \in L^2(\bbR;\cH) \,\bigg|\,
g(t)\in \dom(S(t)) \text{ for a.e.\ } t\in\bbR,    \lb{A.17}  \\
& \quad t \mapsto S(t)g(t) \text{ is (weakly) measurable,} \,
\int_{\bbR} \|S(t) g(t)\|_{\cH}^2 \, dt <
\infty\bigg\}.   \no
\end{align}
An elementary argument shows that if $S(t)$ are closed in $\cH$ for all $t\in\bbR$, 
then $\boldsymbol{\cS}$ is closed in $L^2(\bbR;\cH)$. Indeed, suppose that 
$\{f_n\}_{n\in\bbN} \subset \dom(\bfS)$ such that for some $f, g \in L^2(\bbR;\cH)$, 
\begin{align}
& \lim_{n\to\infty} \big[\|\boldsymbol{\cS} f_n - g\|_{L^2(\bbR;\cH)}^2  
+ \|f_n - f\|_{L^2(\bbR;\cH)}^2\big]   \\
& \quad = \lim_{n\to\infty} \int_{\bbR} dt \, \big[\|S(t) f_n(t) - g(t)\|_{\cH}^2  
+ \|f_n(t) - f(t)\|_{\cH}^2\big] =0.  
\end{align}
Then there exists a subsequence $\{f_{n_{m}}\}_{m\in\bbN}$ of $\{f_n\}_{n\in\bbN}$ 
such that 
\begin{equation}
\lim_{m\to\infty} \big[\|S(t) f_{n_m}(t) - g(t)\|_{\cH}^2  
+ \|f_{n_m}(t) - f(t)\|_{\cH}^2\big] =0 \, \text{ for a.e.\ $t\in\bbR$.}    \lb{A.17a}
\end{equation}
Since by hypothesis $S(t)$ is closed in $\cH$ for all $t\in\bbR$, \eqref{A.17a} 
implies that for a.e.\ $t\in\bbR$, $f(t) \in \dom(S(t))$, $\{f(t)\}_{t\in\bbR}$ is (weakly) measurable in $\cH$ (cf.\ Remark \ref{rA.2}\,$(i)$), and $S(t) f(t) = g(t)$, that is, 
$f \in \dom(\boldsymbol{\cS})$ and 
$\boldsymbol{\cS} f =g$, proving that $\boldsymbol{\cS}$ is closed.

Thus, assuming Hypothesis \ref{hA.0}, one infers that $\boldsymbol{\cT}$, defined according to \eqref{A.17} in terms of the family $\{T(t)\}_{t\in\bbR}$, is closed in 
$L^2(\bbR; \cH)$ (but $\boldsymbol{\cT}$ might not be densely defined, 
cf.\ Example \ref{e4.2} and Remark \ref{r4.5}). If in addition, the 
family $\{T(t)\}_{t\in\bbR}$ is $N$-measurable, then $\boldsymbol{\cT}$ is called 
{\it decomposable} in $L^2(\bbR; \cH) = \int_{\bbR}^{\oplus} \cH \, dt$ and also denoted by the direct integral of the family $\{T(t)\}_{t\in\bbR}$ over $\bbR$ with respect to Lebesgue measure,
\begin{equation}
\bsT = \int_{\bbR}^{\oplus} T(t) \, dt   \lb{A.18}
\end{equation}
(cf.\ also \cite[Ch.\ II]{Di96}, \cite{Ma50}, \cite[Ch.\ I]{Sc67}, \cite{vN49}, 
\cite[Ch.\ XIV]{vN51} 
in the context of bounded operators; \cite{Ch70}, \cite{Le74a} in connection with  spectral operators; \cite{Gi73} for a functional calculus of decomposable operators).
In this case, one also has
\begin{equation}
\bsP(\Gamma(\bsT))_{j,k} = \int_{\bbR}^{\oplus} P(\Gamma(T(t)))_{j,k} \, dt, \quad
j,k \in \{1,2\}.   \lb{A.19}
\end{equation}

We recall once more our convention throughout this manuscript that boldface symbols, such as $\bsT$, denote operators in the Hilbert space $L^2(\bbR;\cH)$ associated with the direct integral over the family $\{T(t)\}_{t\in\bbR}$ as depicted in \eqref{A.18} (as opposed to our choice of notation $\boldsymbol{\cT}$ in the context of \eqref{A.17}). 

If $T(t) \in \cB(\cH)$, $t\in\bbR$, then
\begin{equation}
\boldsymbol{\cT} \in \cB(L^2(\bbR; \cH)) \Longleftrightarrow 
{\rm esssup}_{t\in\bbR} \|T(t)\|_{\cB(\cH)}
< \infty,
\end{equation}
in particular, if $\boldsymbol{\cT} \in \cB(L^2(\bbR; \cH))$, then
\begin{equation}
\|\boldsymbol{\cT} \|_{\cB(L^2(\bbR; \cH))} = {\rm esssup}_{t\in\bbR} \|T(t)\|_{\cB(\cH)}.
\end{equation}

We recall the following results of Nussbaum \cite{Nu64} (in fact, he deals with the more general situation
where the constant fiber space $\cH$ is replaced by a measurable family of Hilbert spaces $\{\cH(t)\}_{t\in\bbR}$):

\begin{lemma} [Nussbaum \cite{Nu64}]  \lb{lA.5}
Assume Hypothesis \ref{hA.0} and suppose in addition that the family
$\{T(t)\}_{t\in\bbR}$ is weakly measurable. Define $\boldsymbol{\cT}$ according to 
\eqref{A.17}. Then $\boldsymbol{\cT}$ is a closed, decomposable operator in
$L^2(\bbR; \cH) = \int_{\bbR}^{\oplus} \cH \, dt$. Thus, there exists an $N$-measurable
family of closed operators $\big\{\hatt T(t)\big\}_{t\in\bbR}$ in $\cH$ such that
\begin{equation}
\boldsymbol{\cT} = \int_{\bbR}^{\oplus} \hatt T(t) \, dt
\end{equation}
and
\begin{equation}
\hatt T(t) \subseteq T(t) \, \text{ for a.e.\ $t\in\bbR$.}
\end{equation}
\end{lemma}

We note that it is not known if $\hatt T(t)$ in Lemma \ref{lA.5} are densely defined
for a.e.\ $t\in\bbR$ in $\cH$ (in which case also $\bfT$ would be densely defined in
$L^2(\bbR; \cH)$; see also Remark \ref{r4.5}).

\begin{theorem} [Nussbaum \cite{Nu64}]  \lb{tA.6}
Assume Hypothesis \ref{hA.0} and suppose in addition that the family
$\{T(t)\}_{t\in\bbR}$ is $N$-measurable. Then the following assertions hold: \\
$(i)$ $\bsT = \int_{\bbR}^{\oplus} T(t) \, dt$ is densely defined and closed in
$L^2(\bbR; \cH) = \int_{\bbR}^{\oplus} \cH \, dt$ and
\begin{equation}
\bsT^* = \int_{\bbR}^{\oplus} T(t)^* \, dt, \quad
|\bsT| = \int_{\bbR}^{\oplus} |T(t)| \, dt.      \lb{A.20}
\end{equation}
$(ii)$ $\bsT$ is symmetric $($resp., self-adjoint, or normal\,$)$ if and only if $T(t)$ is
symmetric $($resp., self-adjoint, or normal\,$)$ for a.e.\ $t\in\bbR$. \\
$(iii)$ $\ker(\bsT) = \{0\}$ if and only if $\ker(T(t)) = \{0\}$ for a.e.\ $t\in\bbR$. In addition,
if $\ker(\bsT) = \{0\}$, then $\big\{T(t)^{-1}\big\}_{t\in\bbR}$ is $N$-measurable and 
\begin{equation}
\bsT^{-1} = \int_{\bbR}^{\oplus} T(t)^{-1} \, dt.    \lb{A.21}
\end{equation}
$(iv)$ If $\bsT$ is self-adjoint in $L^2(\bbR; \cH)$, then $\bsT \geq 0$ if and only
if $T(t) \geq 0$ for a.e.\ $t\in\bbR$. \\
$(v)$ If $\bsT$ is normal in $L^2(\bbR; \cH)$, then
\begin{equation}
p(\bsT) = \int_{\bbR}^{\oplus} p(T(t)) \, dt    \lb{A.22}
\end{equation}
for any polynomial $p$. \\
$(vi)$ Let $S(t)$, $t\in\bbR$, be densely defined, closed operators in $\cH$ and assume
that the family $\{S(t)\}_{t\in\bbR}$ is $N$-measurable and
$\bsS = \int_{\bbR}^{\oplus} S(t)\, dt$. Then $\bsT \subseteq \bsS$ if and only if
$T(t) \subseteq S(t)$ for a.e.\ $t\in\bbR$.
\end{theorem}

Since $N$-measurability is a crucial hypothesis in Theorem \ref{tA.6}, we emphasize
Remark \ref{rA.2}\,$(iii)$ which represents necessary and sufficient conditions which
seem verifiable in practical situations. In addition, we note the following result:

\begin{lemma} \lb{lA.6}
Assume Hypothesis \ref{hA.0} and suppose that
\begin{equation}
\{T(t)\}_{t\in\bbR}, \quad \big\{\big(|T(t)|^2 + I_{\cH}\big)^{-1}\big\}_{t\in\bbR}, \, \text{ and } \,
\big\{T(t)\big(|T(t)|^2+I_{\cH}\big)^{-1}\big\}_{t\in\bbR}
\end{equation}
are weakly measurable. Then $\{T(t)\}_{t\in\bbR}$ is $N$-measurable.
\end{lemma}
\begin{proof}
Since $T(t)\big(|T(t)|^2+I_{\cH}\big)^{-1} \in \cB(\cH)$, $t\in\bbR$, and
\begin{equation}
\big(T(t) \big(|T(t)|^2+I_{\cH}\big)^{-1}\big)^* = T(t)^* \big(|T(t)^*|^2+I_{\cH}\big)^{-1},
\quad t\in\bbR,
\end{equation}
one concludes that $\big\{T(t)^* \big(|T(t)^*|^2+I_{\cH}\big)^{-1}\big\}_{t\in\bbR}$ is weakly
measurable too. Thus,
for each $g \in \cH$, $\big\{T(t)^* \big(|T(t)^*|^2+I_{\cH}\big)^{-1} g\big\}_{t\in\bbR}$ is weakly
measurable in $\cH$, in addition, $T(t)^* \big(|T(t)^*|^2+I_{\cH}\big)^{-1} g \in \dom (T(t))$
for all $t\in\bbR$. Since
$\{T(t)\}_{t\in\bbR}$ is weakly measurable, one thus concludes that
\begin{equation}
\big\{T(t) T(t)^* \big(|T(t)^*|^2+I_{\cH}\big)^{-1}\big\}_{t\in\bbR}
= \big\{I_{\cH} - \big(|T(t)^*|^2+I_{\cH}\big)^{-1}\big\}_{t\in\bbR},
\end{equation}
and hence $\big\{\big(|T(t)^*|^2+I_{\cH}\big)^{-1}\big\}_{t\in\bbR}$, is weakly measurable as well.
 \end{proof}

Next, we recall a result due to Lennon \cite{Le74} on sums and products of decomposable operators (actually, Lennon considers a slightly more general situation). We use the usual conventions that if $A$ and $B$ are linear operators in $\cH$ then
\begin{equation}
\dom(A+B) = \dom(A) \cap \dom(B)
\end{equation}
and
\begin{equation}
\dom(AB) = \{f \in \dom(B) \,|\, Bf \in \dom(A)\}.
\end{equation}

\begin{theorem} [Lennon \cite{Le74}]  \lb{lA.7}
Let $\bsA = \int_{\bbR}^{\oplus} A(t)\, dt$ and $\bsB = \int_{\bbR}^{\oplus} B(t)\, dt$ 
be closed decomposable operators in $L^2(\bbR; \cH) = \int_{\bbR}^{\oplus} \cH \, dt$ with the $N$-measurable families $\{A(t)\}_{t\in\bbR}$ and $\{B(t)\}_{t\in\bbR}$ 
in $\cH$ satisfying Hypothesis \ref{hA.0}. Then the following holds: \\
$(i)$ $\dom(\bsA + \bsB)$ is dense in $L^2(\bbR; \cH)$ if and only if 
$\dom(A(t)\cap B(t))$
is dense in $\cH$ for a.e.\ $t\in\bbR$. In addition, $\bsA + \bsB$ is closable  
in $L^2(\bbR; \cH)$ if and only if $A(t) + B(t)$ is closable in $\cH$ for a.e.\ $t\in\bbR$. In this case the family
$\big\{\ol{[A(t) + B(t)]}\big\}_{t\in\bbR}$ is $N$-measurable and
\begin{equation}
\ol{\bsA + \bsB} = \int_{\bbR}^{\oplus} \ol{[A(t) + B(t)]} \, dt.
\end{equation}
$(ii)$ $\dom(\bsA \bsB)$ is dense in $L^2(\bbR; \cH)$ if and only if $\dom(A(t) B(t))$
is dense in $\cH$ for a.e.\ $t\in\bbR$. In addition, $\bsA \bsB$ is closable in 
$L^2(\bbR; \cH)$ if and only if $A(t) B(t)$ is closable in $\cH$ for a.e.\ $t\in\bbR$. 
In this case the family
$\big\{\ol{[A(t) B(t)]}\big\}_{t\in\bbR}$ is $N$-measurable and
\begin{equation}
\ol{\bsA \bsB} = \int_{\bbR}^{\oplus} \ol{[A(t) B(t)]} \, dt.
\end{equation}
\end{theorem}

We continue this discussion by reconsidering the special self-adjoint case in more
detail which has important applications to periodic Schr\"odinger (resp., Jacobi, Dirac, 
CMV, etc.) operators (cf.\ \cite[Chs.\ 3, 4]{Ku93}, ]\cite[Sect.\ XIII.16]{RS78}) and to random Hamiltonians in condensed matter physics (cf.\ \cite[Ch.\ V, VII--IX ]{CL90}, 
\cite[Chs.\ I, II, VI]{PF92}).

In Theorem \ref{tA.7} below, $E_T(\lambda)$, $\lambda \in \bbR$, denotes the family of strongly right continuos spectral projections associated with the self-adjoint operator $T$ in $\cH$. 

\begin{theorem} \lb{tA.7}
Let $T(t)$, $t\in\bbR$, be self-adjoint operators in $\cH$. Then the following items
$(i)$--$(v)$ are equivalent: \\
$(i)$ The family $\{T(t)\}_{t\in\bbR}$ is $N$-measurable. \\
$(ii)$ For some $($and hence for all\,$)$ $z_0\in\bbC\backslash\bbR$,
$\{(T(t) - z_0 I_{\cH})^{-1}\}_{t\in\bbR}$ is weakly measurable. \\
$(iii)$ For all $\lambda \in \bbR$, $\{E_{T(t)}(\lambda)\}_{t\in\bbR}$ is weakly
measurable. \\
$(iv)$ For all $s \in \bbR$, $\{e^{i s T(t)}\}_{t\in\bbR}$ is weakly measurable. \\
$(v)$ For all $F \in L^\infty(\bbR; dx)$, $\{F(T(t))\}_{t\in\bbR}$ is weakly measurable.
\end{theorem}
\begin{proof}
Since by hypothesis, $T(t) = T(t)^*$, $t\in\bbR$, equation \eqref{A.14A} yields that 
\begin{align}
\begin{split}
& \text{$N$-measurability of $\{T(t)\}_{t\in\bbR}$ is equivalent to weak measurability of} 
\\
& \quad \big\{\big(T(t)^2 + I_{\cH}\big)^{-1}\big\}_{t\in\bbR} \, \text{ and } \, 
\big\{T(t) \big(T(t)^2 + I_{\cH}\big)^{-1}\big\}_{t\in\bbR}.    \lb{A.29} 
\end{split} 
\end{align}
To show the equivalence of items $(i)$ and $(ii)$ we will employ the identities
\begin{align}
\big(T(t)^2 + I_{\cH}\big)^{-1} &= (T(t) - i I_{\cH})^{-1} (T(t) + i I_{\cH})^{-1},  \no \\
T(t)  \big(T(t)^2 + I_{\cH}\big)^{-1} &= (T(t) - i I_{\cH})^{-1} - i \big(T(t)^2 + I_{\cH}\big)^{-1} \lb{A.30} \\
& = (T(t) + i I_{\cH})^{-1} + i \big(T(t)^2 + I_{\cH}\big)^{-1}.  \no
\end{align}
Assuming that the family $\{T(t)\}_{t\in\bbR}$ is $N$-measurable, the identity
\begin{equation}
(T(t) \pm i I_{\cH})^{-1} = T(t)  \big(T(t)^2 + I_{\cH}\big)^{-1} 
\mp i  \big(T(t)^2 + I_{\cH}\big)^{-1},
\end{equation}
then proves that the families $\{(T(t) \pm i I_{\cH})^{-1}\}_{t\in\bbR}$ are weakly measurable. Using the resolvent equation for $T(t)$ and the fact that by \eqref{A.14b},  products of families of weakly measurable bounded operators are weakly measurable, one obtains that for all $z \in \bbC\backslash\bbR$, the family 
$\{(T(t) - z I_{\cH})^{-1}\}_{t\in\bbR}$ is weakly measurable. In particular, this proves 
that $(i)$ implies $(ii)$. Using once more the fact that products of families 
of weakly measurable bounded operators are weakly measurable, combining 
\eqref{A.29} and \eqref{A.30} immediately yields that $(ii)$ implies $(i)$.

The equivalence of
items $(iii)$, $(iv)$, and $(v)$ with item $(ii)$ is familiar from the theory of random
Hamiltonians (cf., e.g., \cite[Sect.\ 5.1]{CL90} and \cite{KM82})
and follows from the facts that for all $f, g \in \cH$, $t\in\bbR$,
\begin{align}
& \big(f, e^{i s T(t)} g\big)_{\cH} = \int_{\bbR} e^{i s \lambda} d(f, E_{T(t)} (\lambda) g)_{\cH},
\quad s\in\bbR,   \\
& \big(f, (T(t) -z I_{\cH})^{-1} g\big)_{\cH}
= \pm i \int_0^\infty e^{\pm i z s} \big(f, e^{\mp i s T(t)} g\big)_{\cH}\, ds, 
\quad \pm \Im (z) > 0, \\
& (f, E_{T(t)} (\lambda) g)_{\cH} = 
\lim_{\delta \downarrow 0} \lim_{\varepsilon \downarrow 0}
\int_{-\infty}^{\lambda + \delta} \big(f,\big[(H(t) - (\lambda' + i \varepsilon) I_{\cH})^{-1}   \\
& \hspace*{5.1cm} - (H(t) - (\lambda' - i \varepsilon) I_{\cH})^{-1}\big] g\big)_{\cH} 
\, d \lambda',
\quad \lambda \in \bbR,    \no
\end{align}
and the fact that $F_{\pm} = (|F| \pm F)/2$ (away from a set of Lebesgue
measure zero) is the limit of appropriate step functions $F_{n,\pm}$, $n\in\bbN$,
in the $\|\cdot\|_{\infty}$-norm, with $0 \leq F_{n,\pm} \leq F_{\pm}$. Finally, one
uses the fact that pointwise limits of measurable functions are measurable.
\end{proof}

The equivalence of items $(i)$ and $(ii)$ in Theorem \ref{tA.7} is of course
well-known and used, for instance, in \cite[Sect.\ XIII.16]{RS78} to define
measurability of $\{T(t)\}_{t\in\bbR}$ in the self-adjoint case.

We conclude with the following elementary yet useful result (cf., e.g.,
\cite[Sect.\ 5.1]{CL90} and \cite{KM82}):

\begin{lemma} \lb{lA.8}
Let $T(t)$ and $T_n(t)$, $n\in\bbN$, be self-adjoint operators in $\cH$ for each
$t\in\bbR$. In addition, suppose that $\{T_n(t)\}_{t\in\bbR}$ is $N$-measurable
in $\cH$ for each $n\in\bbN$, and that for a.e.\ $t \in \bbR$, $T_n (t)$ converge in
the weak $($and hence strong\,$)$ resolvent sense to $T(t)$ as $n\to \infty$. Then 
the  family $\{T(t)\}_{t\in\bbR}$ is $N$-measurable in $\cH$.
\end{lemma}
\begin{proof}
This follows from Theorem \ref{tA.7}\,$(i), (ii)$ and
\begin{equation}
\big(f, (T(t) -z I_{\cH})^{-1} g\big)_{\cH} 
= \lim_{n\to\infty} \big(f, (T_n(t) -z I_{\cH})^{-1} g\big)_{\cH}
\end{equation}
for all $f, g \in \cH$, and a.e.\ $t\in\bbR$, using the fact that the a.e.\ limit of measurable functions is measurable (using the completeness of the Lebesgue 
measure).
\end{proof}

\section{The Negative Answer to Nusssbaum's Question}  \lb{s4}

On p.\ 36 in his paper \cite{Nu64}, Nussbaum asks whether the converse 
implication holds in \eqref{A.14a}. Given the implication in \eqref{A.14a}, this 
amounts of course to asking whether there is actually equivalence in \eqref{A.14a}. Explicitly, this then reads as follows: 
\begin{equation}
\text{$\bullet$ Is $N$-measurability of $\{T(t)\}_{t\in\bbR}$ equivalent to 
weak measurability of $\{T(t)\}_{t\in\bbR}$\,?}    \lb{4.1}
\end{equation}
As described in \eqref{A.14c}, Nussbaum \cite{Nu64} 
proved the equivalence in \eqref{4.1} in the case of bounded operators 
$T(t) \in \cB(\cH)$, $t\in\bbR$, and, as mentioned in Remark \ref{rA.2}\,$(ii)$, he also proved \eqref{A.14a}, and hence that $N$-measurability 
of $\{T(t)\}_{t\in\bbR}$ implies weak measurability of $\{T(t)\}_{t\in\bbR}$. However, 
in the case of unbounded operators $T(t)$, the converse of this fact, and hence the question \eqref{4.1}, appears to have been open for about 45 years now.

The principal purpose of this section is to answer Nussbaum's question and 
consider closely related questions of this type. In fact, we will provide several 
counterexamples with slightly varying degree of sophistication. In addition, we also derive some concrete scenarios in which $N$-measurability can be verified. 
 
The following simple (and most likely, simplest) counterexample answers Nussbaum's question in the negative and hence demonstrates an interesting distinction between operators in $L^2(\bbR; \cH)$ defined as in \eqref{A.17}, and direct integrals as in 
\eqref{A.18}:

\begin{example}  \lb{e4.1}
Let $T_0$ and $T_1$ be densely defined, closed, unbounded, symmetric operators in $\cH$ satisfying
\begin{equation}
T_0 \subsetneq T_1.   \lb{4.2}
\end{equation}
Let $\gE\subset \bbR$ be a nonmeasurable subset of $\bbR$ 
$($in the sense of Lebesgue measure$)$ and introduce the linear operators
\begin{equation}
\wti T(t) = \begin{cases} T_0, & t \in\gE,  \\
T_1, & t \in \bbR\backslash\gE,
\end{cases}      \lb{4.3}
\end{equation}
in $\cH$. Then the family $\big\{\wti T(t)\big\}_{t\in\bbR}$ is weakly measurable, but not
$N$-measurable.
\end{example}
\begin{proof}
Let $\{f(t)\}_{t\in\bbR}$ be a (weakly) measurable family of elements  in $\cH$ such that
$f(t) \in \dom\big(\wti T(t)\big)$ for all $t\in\bbR$. Then, using the fact that
\begin{equation}
T_0 \subset T_1 \subseteq T_1^* \subset T_0^*,     \lb{4.4}
\end{equation}
one concludes that
\begin{equation}
\big(g, \wti T(t) f(t) \big)_{\cH} = (T_0 g, f(t))_{\cH}, \quad t\in\bbR, \;
g \in \dom(T_0),  \lb{4.5}
\end{equation}
is measurable, and since $\dom(T_0)$ is dense in $\cH$, the family
$\big\{\wti T(t)\big\}_{t\in\bbR}$ is weakly measurable by Remark \ref{rA.2}\,$(i)$.

Since by hypothesis, $T_0 \subsetneq T_1$, one concludes that
\begin{equation}
T_0^* T_0 \neq T_1^* T_1.     \lb{4.5a}
\end{equation}
Arguing by contradiction, the equality $T_0^* T_0 = T_1^* T_1$ would imply
$\dom(T_0^* T_0) = \dom(T_1^* T_1)$ and hence
\begin{equation}
\dom(T_0) = \dom(|T_0|) = \dom(|T_1|) = \dom(T_1),   \lb{4.10}
\end{equation}
where we used $\dom(T) = \dom(|T|)$ for any densely defined, closed operator $T$
in $\cH$ and \cite[Theorem\ 9.4(b)]{We80} (a consequence of Heinz's inequality) to obtain $\dom(|T_0|) = \dom(|T_1|)$. However, \eqref{4.10} contradicts hypothesis
\eqref{4.2}.

In addition, there exists $0 \neq h \in \cH$ such that
\begin{equation}
\big(h, (T_0^* T_0 + I_{\cH})^{-1} h\big)_{\cH}
\neq \big(h, (T_1^* T_1 + I_{\cH})^{-1} h\big)_{\cH}.   \lb{4.6}
\end{equation}
Indeed, if no such $0 \neq h \in \cH$ existed, that is, if
\begin{equation}
\big(k, (T_0^* T_0 + I_{\cH})^{-1} k\big)_{\cH}
= \big(k, (T_1^* T_1 + I_{\cH})^{-1} k\big)_{\cH} \, \text{ for all } \, k \in\cH,    \lb{4.7}
\end{equation}
one arrives at
\begin{equation}
\big(k_1, (T_0^* T_0 + I_{\cH})^{-1} k_2\big)_{\cH}
= \big(k_1, (T_1^* T_1 + I_{\cH})^{-1} k_2\big)_{\cH} \, \text{ for all } \, k_1, k_2 \in\cH,
\lb{4.8}
\end{equation}
since for any densely defined, closed, linear operator $T$ in $\cH$, the sequilinear form
$\big(k_1, (T^* T + I_{\cH})^{-1} k_2\big)_{\cH}$, $k_1, k_2 \in \cH$, satisfies
\begin{align}
\big(k_1, (T^* T + I_{\cH})^{-1} k_2\big)_{\cH}
& = \f{1}{4}\big[\big((k_1+ k_2), (T^* T + I_{\cH})^{-1} (k_1+ k_2)\big)_{\cH}  \no \\
& \quad - \big((k_1- k_2), (T^* T + I_{\cH})^{-1} (k_1- k_2)\big)_{\cH}    \no \\
& \quad + i \big((k_1- i k_2), (T^* T + I_{\cH})^{-1} (k_1- i k_2)\big)_{\cH}      \lb{4.9} \\
& \quad - i \big((k_1+ i k_2), (T^* T + I_{\cH})^{-1} (k_1+ i k_2)\big)_{\cH}\big],  \no  \\
& \hspace*{5.35cm} k_1, k_2 \in\cH.    \no
\end{align}
Equation \eqref{4.8} implies $T_0^* T_0 = T_1^* T_1$, again contradicting hypothesis
\eqref{4.2}. Thus, one arrives at the validity of \eqref{4.6}.

Since nonmeasurability of $\gE$ is equivalent to nonmeasurability of its characteristic function $\chi_{\gE}$, one similarly infers that
\begin{equation}
\big(h, \big(\big(\wti T(t)\big)^* \wti T(t) + I_{\cH}\big)^{-1} h\big)_{\cH}
= \begin{cases}
\big(h, (T_0^* T_0 + I_{\cH})^{-1} h\big)_{\cH}, & t \in \gE, \\
\big(h, (T_1^* T_1 + I_{\cH})^{-1} h\big)_{\cH}, & t \in \bbR\backslash\gE,
\end{cases}      \lb{4.11}
\end{equation}
is nonmeasurable, implying that the family $\big\{\wti T(t)\big\}_{t\in\bbR}$ is not
$N$-measurable by equation \eqref{A.14A}.
\end{proof}

An elementary concrete situation, illustrating \eqref{4.2}, \eqref{4.4}, and
\eqref{4.5a}, is obtained as follows: 
\begin{example} \lb{e.4.1a}
Given $\cH=L^2((0,1); dx)$, one introduces
\begin{align}
& T_0 = \f{1}{i} \f{d}{dx}, \quad \dom(T_0) = \big\{f \in L^2((0,1); dx) \,\big|\, f \in AC([0,1]); \,
f(0) = f(1) = 0; \no \\
& \hspace*{8.2cm} f' \in L^2((0,1); dx)\big\},   \lb{4.13} \\
& T_1 = \f{1}{i} \f{d}{dx}, \quad \dom(T_1) = \big\{f \in L^2((0,1); dx) \,\big|\, f \in AC([0,1]); \,
f(0) = f(1);   \no \\
& \hspace*{8.2cm} f' \in L^2((0,1); dx)\big\},    \lb{4.14}
\end{align}
where $AC([0,1])$ denotes the set of absolutely continuous functions on $[0,1]$.
Then $T_0$ is densely defined, closed, symmetric, with deficiency indices $(1,1)$,
and $T_1$ is self-adjoint,
\begin{equation}
T_0 \subsetneq T_1 = T_1^* \subsetneq T_0^*,   \lb{4.15}
\end{equation}
where
\begin{equation}
T_0^* = \f{1}{i} \f{d}{dx}, \quad \dom(T_0^*) = \big\{f \in L^2((0,1); dx) \,\big|\, f \in AC([0,1]); \,
 f' \in L^2((0,1); dx)\big\}.   \lb{4.16} \\
\end{equation}
$($See, e.g., \cite[p.\ 141--142]{RS75} for a discussion of these facts.$)$ 
Moreover, one verifies
\begin{align}
& T_0^* T_0 = - \f{d^2}{dx^2}, \quad \dom(T_0^* T_0) = \big\{f \in L^2((0,1); dx) \,\big|\,
f, f' \in AC([0,1]);   \no \\
& \hspace*{4.6cm} f(0) = f(1) = 0; \, f', f'' \in L^2((0,1); dx)\big\},    \lb{4.17} \\
& T_1^* T_1 = T_1^2 = - \f{d^2}{dx^2}, \quad \dom(T_1^* T_1) 
= \big\{f \in L^2((0,1); dx) \,\big|\,
f, f' \in AC([0,1]);   \no \\
& \hspace*{3.2cm} f(0) = f(1), f'(0)= f'(1); \, f', f'' \in L^2((0,1); dx)\big\},    \lb{4.18}
\end{align}
that is, $T_0^*T_0 \geq \pi^2 I_{L^2((0,1); dx)}$ represents the $($self-adjoint\,$)$ 
Dirichlet Laplacian on $[0,1]$, whereas $T_1^* T_1\geq 0$ represents the 
$($self-adjoint\,$)$ periodic Laplacian on $[0,1]$.
\end{example}

Next, recalling the fact that by Remark \ref{rA.2}\,$(iv)$, $N$-measurability 
of $\{T(t)\}_{t\in \bbR}$ is equivalent to $N$-measurability of 
$\{T(t)^*\}_{t\in \bbR}$, we now address the following natural question: 
\begin{align}
\begin{split}
& \text{$\bullet$ Is weak measurability of $\{T(t)\}_{t\in \bbR}$ equivalent to weak 
measurability} \\
& \;\;\; \text{of $\{T(t)^*\}_{t\in \bbR}$\,?} 
\end{split}
\end{align}
We will show that the answer to this question is also negative (and hence independently answering Nussbaum's question in the negative once again).

Let $A_0$ be a densely defined, closed, symmetric operator in
$\cH$ with deficiency indices $(1,1)$.  Let $A_1$ be a self-adjoint
extension of $A_0$ in $\cH$. Define the operators $T_0$ and $T_1$ by
\begin{equation}
T_0=A_0+i I_{\cH}, \quad T_1=A_1+i I_{\cH}, 
\quad \dom(T_j)=\dom(A_j), \; j=0,1.  \lb{4.20} 
\end{equation} 
We note that $\ker (T_1^*)=\{0\}$. On the other hand, there exists a
vector $e\in \dom(T_0^*)$, with $\|e\|_{\cH}=1$, such that
\begin{equation}
T_0^{*}e=0.    \lb{4.21} 
\end{equation} 
Since $T_0 \subsetneq T_1$ one obtains $T_1^* \subsetneq T_0^*$ and hence 
\begin{equation}
e \notin \dom(T_1^*).     \lb{4.21a}
\end{equation}
(Otherwise, $T_1^* e = T_0^* e = 0$ would yield a contradiction.) Define the operator $T_2$ by
\begin{equation}
T_2 f= T_1 f + (e,T_1 f)_{\cH}\, e, \quad f \in \dom(T_2)=\dom(T_1).  \lb{4.22}
\end{equation} 
Since
\begin{equation}
T_2 = K T_1, \, \text{ where } \,  K = I_{\cH} + (e,\cdot)_{\cH}\, e = K^*,    \lb{4.23}
\end{equation} 
and $K, K^{-1} \in \cB(\cH)$, $T_2$ is closed in $\cH$ and 
\begin{equation}
T_2^* = T_1^* K.     \lb{4.23a}
\end{equation} 

\begin{example} \lb{e4.2}  Let $\gE\subset \bbR$ be a nonmeasurable subset 
of $\bbR$ $($in the sense of Lebesgue measure$)$. Given the facts 
\eqref{4.20}--\eqref{4.23}, let the family $\{\wti T(t)\}_{t\in\bbR}$ of closed, densely
defined  operators in $\cH$ be given by
\begin{equation}
\wti T(t) = \begin{cases} T_1^*, & t \in\gE,  \\
T_2^*, & t \in \bbR\backslash\gE. 
\end{cases}      \lb{defnT}
\end{equation}
Then the family $\big\{\wti T(t)\big\}_{t\in\bbR}$ is  weakly
measurable, but the family of adjoint operators 
$\big\{\big(\wti T(t)\big)^*\big\}_{t\in\bbR}$ is not 
weakly measurable.
\end{example}
\begin{proof}
First one notes that
\begin{equation}
\big(\wti T(t)\big)^* = \begin{cases} T_1, & t \in\gE,  \\
T_2, & t \in \bbR\backslash\gE. 
\end{cases}      \lb{defnT*}
\end{equation}
For any $g\in \dom(T_0)$ one has 
\begin{equation}
(e,T_1 g)_{\cH} = (e,T_0 g)_{\cH} = (T_0^{*}e,g)_{\cH} = 0,
\end{equation} 
and
\begin{equation}
T_2 g = T_1 g = T_0 g.       \lb{L22}
\end{equation}
Next, let $\{f(t)\}_{t\in\bbR}$ be a (weakly) measurable family of
elements  in $\cH$ such that $f(t) \in \dom\big(\wti T(t)\big)$
for all $t\in\bbR$. Then for any $g\in \dom(T_0)$ one obtains by
\eqref{defnT*} and \eqref{L22}, that 
\begin{equation}
\big(g, \wti T(t)f(t) \big)_{\cH}=(T_0 g, f(t))_{\cH}, \quad  t\in\bbR,
\end{equation} 
so that $\big(g, \wti T(t)f(t) \big)_{\cH}$, $g\in \dom(T_0)$, is measurable.
Hence, using $\ol{\dom(T_0)}=\cH$ and Remark \ref{rA.2}\,$(i)$, 
one concludes the weak measurability of $\big\{\wti T(t)\big\}_{t\in\bbR}$.

Noting that $\dom\big(\big(\wti T(t)\big)^*\big)=\dom(T_1)$, $t\in\bbR$, and since
$\ker (T_1^*)=\{0\}$, there exists a vector $h_0\in \dom(T_1)$
such that $(e, T_1 h_0)_{\cH} \neq 0$. Then
\begin{equation}
\big(e, \big(\wti T(t)\big)^*h_0 \big)_{\cH} 
= \begin{cases} (e, T_1 h_0)_{\cH}, & t \in\gE,  \\
2 (e, T_1 h_0)_{\cH}, & t \in \bbR\backslash\gE,
\end{cases}      \lb{defnT*e}
\end{equation}
so that $\bbR\ni t \mapsto \big(e, \big(\wti T(t)\big)^* h_0 \big)_{\cH}$ is not measurable. Thus,  $\big\{\big(\wti T(t)\big)^*\big\}_{t\in \bbR}$ is not weakly measurable.
\end{proof}

\begin{remark} \lb{r.4.3} In Example \ref{e4.2} one has 
\begin{equation}
\bigcap_{t\in\bbR}\,\dom \big(\wti T(t)\big)= \dom (T_1^*)\cap \dom (T_2^*),
\end{equation} 
and
\begin{align}
\dom(T^{*}_1)& = \dom(T_1) = \dom(A_1),   \\
\dom(T^{*}_2)& = \{g\in \cH \,|\, [g+(e,g)_{\cH}\,e] \in \dom(T_1^*)\}.
\end{align}
Thus,
\begin{align}
\bigcap_{t\in\bbR}\,\dom\big(\wti T(t)\big) &= \dom(T^{*}_1)\cap
\dom(T^{*}_2)    \no \\
&= \{g \in \dom(T_1^*) \,|\, [g + (e,g)_{\cH} \, e] \in \dom(T_1^*)\}    \no \\
&= \{g \in \dom(T_1^*) \,|\, (e,g)_{\cH} \, e \in \dom(T_1^*)\}    \no \\
&= \{g\in \dom(T_1^*) \,|\, (e, g)_{\cH} = 0\}, 
\end{align} 
since $e \notin \dom(T_1^*)$ by \eqref{4.21a}. Hence, 
\begin{equation}
\codim \, \bigg(\ol{\bigcap_{t\in\bbR}\,\dom\big(\wti T(t)\big)}\bigg) = 1
\end{equation} 
and so $\bigcap_{t\in\bbR}\,\dom\big(\wti T(t)\big)$ is ``almost
 dense''. It cannot be dense, however, if one is interested in
a family $\big\{\big(\wti T(t)\big)^*\big\}_{t\in\bbR}$ that is not weakly measurable 
 as the next statement shows:
\end{remark}

\begin{lemma}\lb{l4.4}
 Suppose that the family $\{T(t)\}_{t\in \bbR}$ of closed,
densely defined operators in $\cH$ is weakly measurable. Assume
in addition, that
\begin{equation}\lb{4.dense}
\ol{\bigcap_{t\in \bbR}\,\dom(T(t))} = \cH.
\end{equation}
Then the family $\{T(t)^*\}_{t\in \bbR}$ is weakly measurable.
\end{lemma}
\begin{proof} Let $\{f(t)\}_{t\in\bbR}$ be a (weakly) measurable family of elements 
in $\cH$ such that
$f(t) \in \dom(T(t)^*)$ for all $t\in\bbR$. Then for any
$g\in \bigcap_{t\in \bbR}\,\dom(T(t))$ one has 
\begin{equation}
(g, T(t)^*f(t))_{\cH} = (T(t)g, f(t))_{\cH}, \quad t\in \bbR,
\end{equation} 
so that $\bbR\ni t \mapsto (g, T(t)^* f(t))_{\cH}$ is measurable by \eqref{A.14aa}. 
In view of \eqref{4.dense} and Remark \ref{rA.2}\,$(i)$, weak measurability
of $\{T(t)^*\}_{t\in \bbR}$ follows.
\end{proof}

Next, we show that condition \eqref{4.dense} in Lemma \ref{l4.4} can be replaced by the condition that $\boldsymbol{\cT}$ is densely defined in $L^2(\bbR; \cH)$:
 
\begin{lemma}\lb{l4.5}
 Suppose that the family $\{T(t)\}_{t\in \bbR}$ of closed,
densely defined operators in $\cH$ is weakly measurable. Assume
in addition, that $\boldsymbol{\cT}$, given by  \eqref{A.17}, is densely defined in 
$L^2(\bbR; \cH)$. Then the family $\{T(t)^*\}_{t\in \bbR}$ is weakly measurable.
\end{lemma}
\begin{proof} 
Let $\{f(t)\}_{t\in\bbR}$ be a (weakly) measurable family of elements 
in $\cH$ such that $f(t) \in \dom( T(t)^*)$ for all $t\in\bbR$. Consider $g \in \cH$, 
$n\in\bbN$, and 
\begin{equation}
\wti g_n\in L^2(\bbR; \cH) \, \text{ with } \, \wti g_n (t) = \chi_{[-n,n]} (t) g, \; 
t\in\bbR, 
\end{equation} 
(we recall that $\chi_{\cM}$ denotes the characteristic function of a subset 
$\cM \subset \bbR$). Since by hypothesis, 
$\ol{\dom(\boldsymbol{\cT})} = L^2(\bbR; \cH)$, 
there exists a sequence 
$\big\{\wti h_j\big\}_{j\in\bbN} \subset \dom(\boldsymbol{\cT})$ such that 
\begin{equation}  
\big\|\wti g_n - \wti h_j\big\|_{L^2(\bbR; \cH)}^2 
= \int_{\bbR} dt \, \big\|\wti g_n(t) - \wti h_j(t)\big\|_{\cH}^2 
\underset{j\to\infty}{\longrightarrow} 0. 
\end{equation}
Consequently, there exists a subsequence $\big\{\wti h_{j_{k}}\big\}_{k\in\bbN}$ 
of $\big\{\wti h_j\big\}_{j\in\bbN}$ such that for a.e.\ $t\in\bbR$, 
\begin{equation}
\big\|\wti g_n(t) - \wti h_{j_{k}}(t)\big\|_{\cH} \underset{k\to\infty}{\longrightarrow} 0, 
\end{equation}
and hence for a.e.\ $t\in [-n,n]$, 
\begin{equation}
\big\|g - \wti h_{j_{k}}(t)\big\|_{\cH} \underset{k\to\infty}{\longrightarrow} 0.  
\end{equation}
As a result, for a.e.\ $t\in\bbR$, 
\begin{align}
\begin{split} 
& \chi_{[-n,n]} (t) \big(T(t) \wti h_{j_{k}}(t), f(t)\big)_{\cH} 
= \chi_{[-n,n]} (t) \big(\wti h_{j_{k}}(t), T(t) ^* f(t)\big)_{\cH}   \\
& \quad \underset{k\to\infty}{\longrightarrow} \chi_{[-n,n]} (t) (h, T(t) ^* f(t))_{\cH}. 
\end{split} 
\end{align}
Since $h\in\cH$ was arbitrary, this yields that 
$\{\chi_{[-n,n]} (t)T(t)^* f(t)\}_{t\in\bbR}$ is weakly measurable in $\cH$. Since also 
$n\in\bbN$ was arbitrary, one concludes that $\{T(t)^* f(t)\}_{t\in\bbR}$ is weakly measurable in $\cH$, which yields the weak measurability of $\{T(t)^*\}_{t\in \bbR}$. 
\end{proof}

\begin{remark}\lb{r4.5}
Lemma \ref{l4.5} shows that if $\{T(t)^*\}_{t\in \bbR}$ is not weakly measurable, 
then $\boldsymbol{\cT}$ is not densely defined. In other words, Example \ref{e4.2}
also generates an example of a non-densely defined operator $\boldsymbol{\cT}$ 
in $L^2(\bbR; \cH)$.
\end{remark}

The following result further clarifies the issue of density of the domain of 
$\boldsymbol{\cT}$.

\begin{lemma} \lb{l4.6} 
Let $\{T(t)\}_{t\in \bbR}$ be a
weakly measurable family of densely defined, closed operators in
$\cH$ and assume that the linear operator $\boldsymbol{\cT}$ on
$L^2(\bbR;\cH)$ is  defined as  in \eqref{A.17}.  If
\begin{equation} \lb{m}
\cD = \big\{f \in L^2(\bbR;\cH) \,\big|\, f(t)\in \dom(T(t)) \, \text{for a.e.} \, t\in\bbR\big\},
\end{equation}
then
\begin{equation}
\ol{\cD}=\ol{\dom(\boldsymbol{\cT})}.
\end{equation} 
\end{lemma}
\begin{proof}
Clearly $\dom(\boldsymbol{\cT})\subseteq \cD$, hence it suffices to prove that
\begin{equation}
\cD \subseteq \ol{\dom (\boldsymbol{\cT})}.
\end{equation} 

If $f\in \cD$, then the family  $\{T(t)f(t)\}_{t \in\bbR}$ is measurable, hence the 
function $\bbR\ni t \mapsto \|T(t)f(t)\|_{\cH}$ is measurable too. The latter fact 
implies that the sets
\begin{equation}
E_n = \{t\in\bbR \,|\, \|T(t)f(t)\|_{\cH} \leq n\} \cap [-n,n], \quad n\in \bbN,
\end{equation} 
are measurable. In addition, one notes that $E_n\subseteq E_{n+1}, n \in \bbN$, and
$\bbR=\bigcup_{n\in\bbN}\,E_n$. Define
\begin{equation}
f_n(t) = \begin{cases} f(t), & t\in E_n,\\
0, & t\in\bbR \backslash E_n, 
\end{cases}
 \lb{by}
\end{equation}
then $f_n\in \cD$, $n\in \bbN$, and
\begin{equation}
\|f-f_n\|_{L^2(\bbR;\cH)}^2 = \int_{\bbR \backslash 
E_n}\,\|f(t)\|_{\cH}^2\,dt \underset{n\to\infty}{\longrightarrow} 0.
\end{equation} 
Moreover, $f_n\in \dom(\boldsymbol{\cT})$ since
\begin{equation}
\int_{\bbR} \|T(t)f_n(t)\|_{\cH}^2\,dt
=\int_{E_n}\,\|T(t)f(t)\|_{\cH}^2\,dt\leq 2n^3.
\end{equation} 
Thus $\cD \subseteq \ol{\dom(\boldsymbol{\cT})}$. 
\end{proof}

\begin{corollary} \lb{c4.7}
Let $\{T(t)\}_{t\in \bbR}$ be a
weakly measurable family of densely defined, closed operators in
$\cH$ and assume that the linear operator $\boldsymbol{\cT}$ on
$L^2(\bbR;\cH)$ is  defined as  in \eqref{A.17}. Then $\boldsymbol{\cT}$ is densely
defined if the linear subspace $\cD$ of  $L^2(\bbR;\cH)$
given by \eqref{m} is dense in $L^2(\bbR;\cH)$. 
\end{corollary}

Strengthening Example \ref{e4.1}, we now provide an example of
operators $\{T(t)\}_{t\in \bbR}$ such that \eqref{4.1} is violated and,
in adddition, $T(t)$ are self-adjoint in $\cH$ for all $t\in \bbR$:

\begin{example}\lb{e4.8} Assume that $T_0$ is a densely defined, closed, symmetric
operator with deficiency indices $(1,1)$ in $\cH$, and let $T_1$, $T_2$, 
$T_1 \neq T_2$, be two different self-adjoint extensions of $T_0$ in $\cH$. 
Suppose that $\gE\subset \bbR$ is a nonmeasurable subset of $\bbR$
$($in the sense of Lebesgue measure$)$. Let $\big\{\wti T(t)\big\}_{t\in\bbR}$ be the 
family of closed, densely defined  operators in $\cH$ given by
\begin{equation}
\wti T(t) = \begin{cases} T_1, & t \in\gE,  \\
T_2, & t \in \bbR\backslash\gE. 
\end{cases}      \lb{defnT1}
\end{equation}
Then the family  $\big\{\wti T(t)\big\}_{t\in \in\bbR}$ of
self-adjoint operators  is weakly measurable, but not
$N$-measurable.
\end{example}
\begin{proof}
We argue as in Example \ref{e4.1}. Consider 
a (weakly) measurable family of elements $\{f(t)\}_{t\in\bbR}$ in
$\cH$ such that $f(t) \in \dom\big(\wti T(t)\big)$ for all
$t\in\bbR$.

Then for any $g\in \dom (T_0)$ one has
\begin{equation}
\big(g, \wti T(t) f(t)\big)_{\cH} = (T_0 g, f(t))_{\cH}, \quad t \in\bbR,
\end{equation} 
and hence $\bbR \ni t \mapsto \big(g, \wti T(t)f(t)\big)_{\cH}$ is measurable. 
Since $\dom(T_0)$ is dense in $\cH$,  Remark \ref{rA.2}\,$(i)$ implies weak measurability of $\big\{\wti T(t)\big\}_{t\in\bbR}$.

On the other hand, as in Example \ref{e4.1},  there exist elements
$h,k\in \cH$ such that
\begin{equation}
\alpha_1 = \big(h, \big(T_1^2+I\big)^{-1}k\big)_{\cH} 
\neq \big(h, \big(T_2^2+I\big)^{-1}k\big)_{\cH} = \alpha_2.
\end{equation} 
Then the function
\begin{equation}
\big(h, \big(\big(\wti T(t)\big)^*\wti T(t) + I_{\cH}\big)^{-1}k\big)_{\cH} 
= \begin{cases} \alpha_1, & t \in\gE,  \\
\alpha_2, & t \in \bbR\backslash\gE,
\end{cases}      \lb{defnT2}
\end{equation}
is nonmeasurable. Thus, the family $\big\{\wti T(t)\big\}_{t\in\bbR}$ is not 
$N$-measurable.
\end{proof}

We conclude with a positive result characterizing
$N$-measurability in terms of resolvents (cf.\ also 
Theorem \ref{tA.7}\,$(i)$ and $(ii)$):

\begin{theorem}  \lb{t4.9} Let $\{T(t)\}_{t\in \bbR}$ be a
family of  densely defined, closed operators in $\cH$. Suppose that
there exists a measurable function $\bbR \ni t \mapsto \alpha(t) \in \bbC$ such that
\begin{equation}
\alpha(t)\in \rho(T(t)) \, \text{ for a.e.\ } t\in\bbR,
\end{equation} 
and, in addition, the family of bounded operators
$\big\{(T(t)-\alpha(t)I_{\cH})^{-1}\big\}_{t\in \bbR}$ is weakly measurable.
Then the family $\{T(t)\}_{t\in \bbR}$ is $N$-measurable.
\end{theorem}
\begin{proof}
Since $(T(t) - \alpha(t) I_{\cH})^{-1} \in \cB(\cH)$, $t\in\bbR$, weak measurability 
of the family $\big\{(T(t) - \alpha(t) I_{\cH})^{-1} \big\}_{t\in\bbR}$ is equivalent 
to its $N$-measurability by \eqref{A.14c}. By Theorem \ref{tA.6}\,$(iii)$, the 
family $\{T(t)-\alpha(t)I_{\cH}\}_{t\in\bbR}$ is $N$-measurable. Since 
$\alpha(t) I_{\cH} \in \cB(\cH)$ and $T(t)$ is closed for all $t\in\bbR$, 
Theorem \ref{lA.7}\,$(i)$ implies that 
$[T(t)-\alpha(t)I_{\cH}] + \alpha(t) I_{\cH}$ is densely defined and closable for 
a.e.\ $t\in\bbR$, and additionally,   
\begin{equation}
\ol{[T(t)-\alpha(t)I_{\cH}] + \alpha(t) I_{\cH}} = \ol{T(t)} = T(t) \, 
\text{ is $N$-measurable.}
\end{equation} 
\end{proof}

\medskip

\noindent {\bf Acknowledgments.}
We are indebted to Vladimir Chilin, Ben de Pagter, Nigel Kalton, and 
Konstantin Makarov for helpful discussions. 
In particular, Nigel Kalton suggested looking for examples of the type 
Example \ref{e4.8} with self-adjoint fiber operators $\wti T (t)$, $t\in\bbR$. 

\medskip


\end{document}